\documentclass[12 pt, a4paper, oneside]{amsart}
\usepackage{indentfirst,amsmath,amssymb,amsthm,fullpage,amsfonts,enumerate,mathrsfs,tabularx,graphicx, xcolor, tikz-cd,hyperref, setspace, upgreek, mathabx, titletoc, mathrsfs, mathtools, float, subfig, amsthm, dsfont,marginnote,pdfcomment,multicol,xfrac,adjustbox,xparse,etoolbox}
\usetikzlibrary{matrix,arrows,backgrounds,positioning}
\makeatletter
\patchcmd{\@thm}{\thm@headfont{\scshape}}{\thm@headfont{\scshape\bfseries}}{}{}
\patchcmd{\@thm}{\thm@notefont{\fontseries\mddefault\upshape}}{}{}{}
\makeatother

\makeatletter
\@addtoreset{figure}{section}
\makeatother

\newtheorem{theorem}{Theorem}[section]
\newtheorem{proposition}[theorem]{Proposition}
\newtheorem{lemma}[theorem]{Lemma}
\newtheorem{corollary}[theorem]{Corollary}
\newtheorem{remark}[theorem]{Remark}
\newtheorem{definition}[theorem]{Definition}

\newtheorem{situation}[theorem]{Situation}

\newcommand{\integers}{\mathbb{Z}}

\newcommand{\rationals}{\mathbb{Q}}

\newcommand{\reals}{\mathbb{R}}

\newcommand{\comment}[1]{}
\newcommand*{\spec}{\text{Spec}}
\newcommand*{\pic}{\text{Pic}}
\renewcommand{\spec}{\text{Spec}}
\newcommand{\etale}{étale}
\renewcommand{\hom}{\text{Hom}}
\newcommand{\ext}{\text{Ext}}

\newcommand{\etalecohomology}[3]{\text{H}^{#1}_{\acute{e}t}(#2,#3)}

\newcommand{\homology}[3]{\text{H}_{#1}(#2,#3)}
\newcommand{\cohomology}[3]{\text{H}^{#1}(#2,#3)}

\newcommand{\structuresheaf}{\mathcal{O}}
\newcommand*{\sheafhom}{\mathop{\mathcal{H}\mathit{om}}}

\newcommand{\structure}{\mathcal{O}}

\tikzset{close/.style={outer sep=-2pt}}

\title{Weil pairing on twisted curves}
\author{Ashwin Deopurkar}
\address{Ashwin Deopurkar}
\email{ard2144@columbia.edu}
\date{\today} 

\begin{document}    
\maketitle 
\NewDocumentCommand\gmsheaf{o}{\IfNoValueTF{#1}{\mathbf{G}_{m}}{\mathbf{G}_{m,{#1}}}}
\section{Introduction}
{
\newcommand{\autgroup}{\integers/ 2 \integers}
\newcommand{\twistedcurve}{\mathcal{X}}
\newcommand{\groupstack}{M_g^{\structuresheaf,2}}
\newcommand{\groupstackcompact}{\overline{M_g}^{\structuresheaf,2}}
\begin{sloppypar}
Let $C$ be a smooth, proper, connected curve of genus $g \geq 2$ over an algebraically closed field of characteristic different from two. We have the Weil pairing on $C$ which is a perfect, alternating pairing on $\pic(C)[2]$. It is natural to consider this pairing as $C$ varies in the moduli space $M_g$ of smooth, proper curves. The natural object to consider is the moduli space 
  $$\groupstack = \left. \{(C,L,f) | f: L^{\otimes 2} \to \structuresheaf_{C} \} \middle/ \cong \right. $$ 
  that parameterizes smooth, proper curves of genus $g$ together with a $2\mbox{-torsion}$ line bundle. The space $\groupstack$ is a Deligne-Mumford stack that forms a finite \etale{} cover of $M_g$ of degree $2^{2g}$ (See \cite{chiodo_stable_twisted_and_r_spin}). In this setting, we can realize the Weil pairing as a pairing on the fibers of $\groupstack$ over $M_g$. A natural question to ask is whether we can extend the pairing on a compactification of $M_g$. We may consider the category $\groupstackcompact$ that parameterizes stable curves together with a two torsion line bundle. It is an \etale{} Deligne-Mumford stack over the space $\overline{M_g}$ of stable curves; however, the morphism ${\groupstackcompact \to \overline{M_g}}$ is not proper (See Example 1.1, \cite{chiodo_stable_twisted_and_r_spin}). Thus, we fail to obtain a finite \etale{} cover of $\overline{M_g}$. This defect is rectified by Chiodo using the category of Abramovich and Vistoli's twisted curves (See \cite{chiodo_stable_twisted_and_r_spin}). We have a compact moduli of twisted curves over which the two torsion line bundles form a finite, \etale{} group sheaf.

  In the present work, we extend the Weil pairing to the moduli of twisted curves. Let $\twistedcurve$ be a twisted curve with the coarse space $C$. Let $H$ denote the subgroup of $\pic(C)[2]$ of those line bundles which are trivial on the normalization of $C$. We prove that $H$ is naturally Weil dual to the quotient $\pic(\twistedcurve)[2]/\pic(C)[2]$. Both $H$ and $\pic(\twistedcurve)[2]/\pic(C)[2]$ have a combinatorial description in terms of the dual graph of $\twistedcurve$. Moreover, the Weil pairing between the two can be seen as a combinatorial integration-homology pairing on the dual graph (Theorem \ref{thm_combinatorial_pairing_induced_by_Weil}).

In the last section, we look at the interplay of tropical and algebraic geometry that emerges when we have an arithmetic surface. We give an algebro-geometric argument to prove that the kernel of the tropical specialization is isotropic for the Weil pairing when the dual graph is totally degenerate (Proposition \ref{prop_ker_trop_isotropic}). This was proved earlier using tropical methods (See \cite{Jensen_Yoav_kernel_is_isotropic}). Finally, we look at arithmetic surfaces where the special fiber is a twisted curve. We obtain a combinatorial consequence that the group of $r\mbox{-torsion}$ points of the Picard group of a metric graph with unit edge lengths can be realized as divisors supported on a set which is the union of the original vertices of the graph and interior vertices that subdivide each non-separating edge into $r$ equal parts (Proposition \ref{prop_divide_edges_into_r_parts}).
\end{sloppypar}
}
\section{Preliminaries}
{
\newcommand{\twistedcurve}{\mathcal{X}} 
\newcommand{\character}[2]{{\theta}^{#1}_{#2}}
\newcommand{\bgspace}[1]{\text{B}#1}
\newcommand{\stab}[1]{G_{#1}}

Let $k$ be an algebraically closed field. For a finite group $G$, we denote by $\bgspace{G}$ the quotient stack $\left[ \spec(k) \middle/ G \right]$. By $\gmsheaf[X]$ we denote the sheaf of invertible functions on $X$ on the \etale{} site of $X$. We omit the subscript and simply write $\gmsheaf$ when the object is clear from the context. Similarly, by $\mu_r$ we mean the sheaf of $r^{th}$ roots of unity. We assume all stacks to be separated.  

Let $M,N$ be modules over a ring $R$. A pairing on $M \times N$ is a $R\mbox{-bilinear}$ map ${M \times N \xrightarrow{e(-,-)} R}$. By the kernel of a pairing in the first factor we mean the kernel of the associated map ${M \to \hom_{R}(N,R)}$. We say a pairing is perfect, if this map is an isomorphism. We say a pairing is non-degenerate if the kernel of the pairing in both the factors is zero. By a pairing on $M$, we mean a pairing on ${M \times M}$. We say a pairing on $M$ is alternating if we have ${e(m,m) = 0}$ for every $m \in M$. It is straightforward to see that for an alternating pairing, we have ${e(m_1,m_2) = -e(m_2,m_1)}$. If we have a pairing on $M$, then a submodule $M' \subset M$ is said to be isotropic, if the pairing restricted to $M'$ is identically zero. 

By a graph we mean a finite, undirected graph with a finite number of vertices and edges. We allow multiple edges between two vertices as well as loop edges. 

\begin{definition}\label{def_orbifold_curve_definition}
A smooth orbifold curve of genus $g$ over $k$ is a connected, proper, smooth, tame Deligne-Mumford stack of dimension $1$ over $k$ which has trivial generic stabilizer and whose coarse space is a smooth, proper curve of genus $g$.
\end{definition}
\begin{sloppypar}
  Let ${\twistedcurve \xrightarrow{f} C}$ be a smooth orbifold curve over $k$ with coarse space $C$. It is known that $\twistedcurve$ is obtained from its coarse space by a finite number of root constructions (See chapter 10.3, \cite{olsson_algebraic_spaces_stacks_book} for root stacks). That is, we have
$$ \twistedcurve \cong C[p_1 / d_1] \times_{C} C[p_2 / d_2] \times_{C} \cdots \times_{C} C[p_m / d_m] $$
where $\{p_i\}$ are closed points of $C$, $\{d_i\}$ are positive integers coprime with the characteristic of $k$, and $C[p_i/d_i]$ is the root stack with the stabilizer $\integers / d_i \integers$ at $p_i$. We denote the stabilizer $\integers / d_i \integers$ at $p_i$ by $\stab{p_i}$. The \etale{} local structure of $\twistedcurve \xrightarrow{f} C$ at a point $p_i$ is given by:
  $$ \begin{tikzcd}
    \left[ \spec(k[t]) \middle/ G_{p_i} \right] \arrow{r}{f} &    \spec(k[t^{d_i}]) 
  \end{tikzcd}$$
  where the generator of $G_{p_i}$ acts on $\spec(k[t])$ by $t \to \zeta_{d_i} t$ for some primitive $d_i^{th}$ root $\zeta_{d_i}$ of unity. For each point $p_i$ with a non-trivial stabilizer, the orbifold admits a line bundle $L_i$ together with an isomorphism of line bundles ${L_i^{ \otimes d_i} \to f^*(\structuresheaf_{C}(p_i))}$. Like schematic proper curves, we have a notion of a degree of a line bundle on $\twistedcurve$. However, the degree is not necessarily integral but takes values in $1/d \integers$, where $d$ is the least common multiple of the numbers $\{d_i\}$. For example, the degree of $L_i$ equals $1/d_i$. We have a closed immersion 
  ${p_i \times_{C} \twistedcurve \cong \bgspace{\stab{p_i}} \xrightarrow{j} \twistedcurve}$. The Picard group of $\bgspace{\stab{p_i}}$ is isomorphic to $\hom(\stab{p_i},k^*) \cong \integers / d_i \integers$ and is generated by $j^*(L_i)$. Thus we have a morphism ${\character{\twistedcurve}{p_i}: \pic(\twistedcurve) \to \integers / d_i \integers}$  which maps a line bundle $L$ to the unique integer $\character{\twistedcurve}{p_i}(L)$ modulo $d_i$ such that ${j^*(L) \cong j^*(L_i)^{\otimes \character{\twistedcurve}{p_i}(L)}}$. We refer to $\character{\twistedcurve}{p_i}(L)$ as the character of $L$ at $p_i$. The Picard group of $\twistedcurve$ is described by the following exact sequence (See Corollary 4.15, \cite{poma2013etale}):
  $$ 0 \to \pic(C) \xrightarrow{f^*} \pic(\twistedcurve) \xrightarrow{\oplus \character{\twistedcurve}{p_i}} \bigoplus_{i = 1}^m \frac{\integers}{d_i\integers} \to 0 .$$
  \begin{remark}\label{remark_non_integral_degree}
    For a line bundle $L$ on $\twistedcurve$, we have $\deg(L) \in 1/d \integers$ and ${\deg(L) - \Sigma_{i =  1}^{m} \character{\twistedcurve}{p_i}(L)}$ is an integer (See Proposition 2.8, \cite{chiodo_stable_twisted_and_r_spin}).
  \end{remark}
\noindent  We have $\etalecohomology{1}{\twistedcurve}{\gmsheaf} \cong \pic(\twistedcurve)$ and $\etalecohomology{2}{\twistedcurve}{\gmsheaf} = 0 $ (See Corollary 4.15, \cite{poma2013etale}).
\end{sloppypar}
} 

\section{Twisted curves}\label{sec_twisted_curves}
{
\newcommand{\autgroup}{\integers/ 2 \integers}
\newcommand{\curvedown}{\mathcal{X}}
\newcommand{\character}[2]{\theta^{#1}_{#2}}
\newcommand{\classifyingspace}[1]{\text{B}#1} 
\newcommand{\twistedcurve}{\mathcal{X}}
\newcommand{\graphpairing}[3]{\langle #2,#3 \rangle_{#1}}
\newcommand{\stab}[1]{G_{#1}}

\begin{definition}\label{def_of_a_twisted_curve}
  A (balanced) twisted curve of genus $g$ over $k$ is a connected, proper tame Deligne-Mumford stack $\twistedcurve$ over $k$ of dimension one which admits a coarse space $C$ that is a nodal curve of arithmetic genus $g$. The morphism $\twistedcurve \to C$ is an isomorphism except at the nodes of $C$ where the \etale{} local structure of $\twistedcurve$ is given by 
$$ \left[\spec(k[x,y]/xy) \middle/ \frac{\integers}{l \integers} \right]$$ 
where the generator of $\integers/ l \integers$ acts by $x \to \zeta_l x, y \to \zeta_l^{-1}y$ for some primitive $l^{th}$ root $\zeta_l$ of unity.
\end{definition}
\begin{remark}
The last condition that $\integers / l \integers$ acts by inverse characters on the two branches is called balancing condition but we omit this adjective because we always consider twisted curves which are balanced.
\end{remark}
\begin{sloppypar}
  Let $\twistedcurve \xrightarrow{f} C$ be a twisted curve of genus $g$ with the coarse space $C$. Let $E$ denote the set of nodal points of $C$. For a point $e$ in $E$, let $\stab{e}$ denote the stabilizer of $e$. The normalization ${\nu : \widehat{\twistedcurve} \to \twistedcurve}$ is ${\nu : \widehat{\twistedcurve} = \widehat{C}\times_{C}\twistedcurve \to \twistedcurve}$ where ${\widehat{C} \to C}$ is the normalization of $C$. The stack $\widehat{\twistedcurve}$ is a disjoint union of smooth orbifold curves as in Definition \ref{def_orbifold_curve_definition}. The points of $\widehat{\twistedcurve}$ with non-trivial stabilizers are precisely those which lie above the nodes of $C$. The degree of a line bundle $L$ on $\twistedcurve$ is defined as the sum of the degrees of $\nu^*(L)$ on the connected components of the normalization $\widehat{\twistedcurve}$. For $e \in E$, we have a closed immersion ${e \times_{C} \twistedcurve \cong BG_e \to \twistedcurve}$.
  \begin{lemma}\label{lemma_pic_of_twisted_exact_sequence}
    (Proposition 5.2, \cite{poma2013etale}) The Picard group of $\twistedcurve$ is described by the following short exact sequence:
    \begin{align*}  
      0 \to \pic(C) \to \pic(\twistedcurve) \to \bigoplus_{e \in E} \pic(BG_e) \to 0.
    \end{align*}
Moreover, for $e \in E$, the group $\pic(BG_e) \cong \hom(G_e,k^*)$ is a cyclic group of order $\lvert G_e \rvert$.
  \end{lemma}
\end{sloppypar}
\begin{proposition}\label{Prop_degree_is_integer}
Let $L$ be a line bundle on $\twistedcurve$. Then the degree of $L$ is an integer. 
\end{proposition}
\begin{proof}
  \begin{sloppypar}
    This essentially follows from the balancing conditions at the nodes (See Proposition 2.19, \cite{chiodo_stable_twisted_and_r_spin}). For a node $e$ of $C$, we have a closed immersion $BG_e \to \twistedcurve$. The fiber ${BG_e \times_{\twistedcurve} \widehat{\twistedcurve}}$ of the normalization over $BG_e$ is isomorphic to $BG_e  \coprod BG_e$. The morphism ${BG_e  \coprod BG_e \xrightarrow{\nu} BG_e}$ differs on the two copies of $BG_e$ by the automorphism corresponding to the inverse automorphism of the stabilizer $G_e$. Consequently, the characters of $\nu^*(L)$ at the two pre-images of $e$ are inverses of each other. Hence the total degree of $L$ is an integer (See Remark \ref{remark_non_integral_degree}).
  \end{sloppypar}
\end{proof}
\begin{definition}\label{definition-dual-graph}
The dual graph associated to $\twistedcurve$ is a graph that has a vertex for every irreducible component of $\twistedcurve$. For every node in $\twistedcurve$, there is an (undirected) edge which joins the vertices corresponding to the branches that the node belongs to.
\end{definition}
\begin{sloppypar}
  We denote the dual graph of $\twistedcurve$ by $\Gamma$. We denote the vertex set of the dual graph by by $V$ and the edge set is naturally parameterized by $E$. 
\end{sloppypar}
\begin{definition}\label{definition-boundary-in-a-graph}
  \begin{sloppypar}
    By $C_0(\Gamma,\integers/2\integers)$ (resp. $C_1(\Gamma,\integers/2\integers)$), we denote the free $\autgroup$ vector space generated by the vertex (resp. edge) set of $\Gamma$. The boundary map ${C_1(\Gamma,\integers/2\integers) \xrightarrow{\partial} C_0(\Gamma,\integers/2\integers)}$ is defined by ${\partial (e = (v_1, v_2)) = v_2 - v_1}$ We denote the kernel of $\partial$ by $\homology{1}{\Gamma}{\autgroup}$.
  \end{sloppypar}
\end{definition}
\begin{proposition}\label{prop_exact_sequence_picard_two_torsion_of_twisted}
Assume that the stabilizers of the nodes of $\curvedown$ have even order. Then we have the following exact sequence (See Corollary 3.1, \cite{chiodo_stable_twisted_and_r_spin}). 
$$ 0 \to \pic(C)[2] \to \pic(\twistedcurve)[2] \to \text{H}_1(\Gamma, \integers / 2 \integers) \to 0 .$$
In particular, the order $\pic(\twistedcurve)[2]$ of is $2^{2g}$. 
\end{proposition}
\begin{proof}
We have the exact sequence as stated in Lemma \ref{lemma_pic_of_twisted_exact_sequence} that describes the Picard group of $\curvedown$. By applying the squaring map from the sequence to itself we get the following exact sequence:
\[
  \begin{tikzcd}[column sep=tiny]
    0 \arrow{r} & {\pic(C)[2]} \arrow{r} & {\pic(\twistedcurve)[2]} \arrow{r} & \displaystyle\bigoplus_{e \in E} \pic(BG_e)[2] \arrow{r}{\partial'} & \frac{\pic(C)}{\pic(C)^2} \arrow{r} & \frac{\pic(\twistedcurve)}{\pic(\twistedcurve)^2} \arrow{r} & \displaystyle\bigoplus_{e \in E} \frac{\pic(BG_e)}{\pic(BG_e)^2} \arrow{r} & 0
  \end{tikzcd}
\]
Since for all $e$ in $E$, the group $G_e$ is a cyclic group of even order, we can identify $\bigoplus_{e \in E} \pic(BG_e)[2]$ with $C_1(\Gamma,\autgroup)$. Also, we have $\pic(C)/\pic(C)^2 \cong C_0(\Gamma,\autgroup)$. With these identifications, the connecting morphism $\partial'$ in the above long exact sequence fits in the following diagram:
\[
\begin{tikzcd}
  \displaystyle\bigoplus_{e \in E} \pic(BG_e)[2] \arrow{d}{\rotatebox{90}{$\simeq$}} \arrow{r}{\partial'} & \frac{\pic(C)}{\pic(C)^2} \arrow{d}{\rotatebox{90}{$\simeq$}} \\
  C_1(\Gamma,\autgroup) \arrow{r}{\partial} & C_0(\Gamma,\autgroup) 
\end{tikzcd}
\]
Finally, if $g'$ denotes the genus of the graph $\Gamma$, then we have $\lvert \pic(C)[2] \rvert = 2^{2g-g'}$ and $\lvert \homology{1}{\Gamma}{\autgroup} \rvert = 2^{g'}$. This completes the proof.
\end{proof}
\begin{definition}\label{definition_cycle_in_a_graph}
  \begin{sloppypar}
    By a cycle in a graph we mean a sequence $\{e_1, e_2, \ldots , e_l\}$ of distinct edges for which there exists a sequence of distinct vertices $\{v_1, v_2, \ldots, v_l \}$ such that for $1 \leq i \leq l-1$, the edge $e_i$ joins the vertices $(v_i,v_{i+1})$, and the last edge $e_l$ joins the vertices $(v_l,v_1)$. We say an edge $e$ is non-separating, if deleting $e$ does not increase the number of connected components of the graph.
  \end{sloppypar}
\end{definition}
\begin{remark}\label{remark_edge_disjoint_cycles}
  \begin{sloppypar}
It is straightforward to see that a chain ${e_1 + e_2 + \ldots +e_l}$ in $C_1(\Gamma,\integers/2\integers)$ is killed by the boundary map ${\partial: C_1(\Gamma,\integers/2\integers) \to C_0(\Gamma,\integers/2\integers)}$ if and only if we can write the set $\{e_1, e_2, \ldots, e_l\}$ as an edge-disjoint union of cycles. Let $\Gamma'$ denote a graph obtained by deleting some edges from $\Gamma$. We may conclude from the previous observation that the natural injective map ${\homology{1}{\Gamma'}{\autgroup} \to \homology{1}{\Gamma}{\autgroup}}$ is an isomorphism if and only if $\Gamma'$ contains all the non-separating edges of $\Gamma$. When not all stabilizers are of even order, we have the exact sequence:
$$  0 \to \pic(C)[2] \to \pic(\twistedcurve)[2] \to \text{H}_1(\Gamma', \integers / 2 \integers) \to 0 $$
where $\Gamma'$ is the graph obtained by deleting the edges of $\Gamma$ that correspond to the nodes with odd stabilizers. This can easily be seen from the proof of the above Proposition \ref{prop_exact_sequence_picard_two_torsion_of_twisted}. Thus we have ${\lvert 
    \pic(\twistedcurve)[2] \rvert = 2^{2g}}$ if and only if the stabilizers of the non-separating nodes have even order (See Theorem 3.9, \cite{chiodo_stable_twisted_and_r_spin}).
  \end{sloppypar}
\end{remark}
\begin{definition}\label{definition-coboundary-map}
  \begin{sloppypar}
    By $C^0(\Gamma,\autgroup)$ (resp. $C^1(\Gamma,\integers/2\integers)$), we denote the vector space of $\integers/2\integers$ valued functions on the vertex (resp. edge) set of $\Gamma$. The coboundary map ${C^0(\Gamma,\integers/2\integers) \xrightarrow{\delta} C^1(\Gamma,\integers/2\integers)}$ is defined by ${\delta f(e = (v_1, v_2)) = f(v_2) - f(v_1)}$. We denote the cokernel of this map by $\cohomology{1}{\Gamma}{\autgroup}$.
  \end{sloppypar}
\end{definition}
\begin{proposition}\label{prop_two_torsion_trivial_on_normalization}
  \begin{sloppypar}
    We set ${H = \{ L \in \pic(\twistedcurve)[2] \, \vert \, \nu^*(L) \text{ is trivial on all components of } \widehat{\twistedcurve}. \}}$. Then we have an isomorphism ${H \cong \cohomology{1}{\Gamma}{\autgroup}}$.
  \end{sloppypar}
\end{proposition}
\begin{proof}
  \begin{sloppypar}
    We note that if ${L \in H}$ then the characters of $\nu^*(L)$ at all the stacky points of $\widehat{\twistedcurve}$ are trivial. Therefore, $L$ is isomorphic to $f^*(L')$ for some line bundle $L' \in \pic(C)[2]$ such that $L'$ pulls back to a trivial line bundle on each component of the normalization $\widehat{C}$. We have the following exact sequence:
\[
  \begin{tikzcd}
    0 \arrow{r} & {\left. (  \displaystyle\bigoplus_{i \in V} k^* ) \middle/ k^* \right.} \arrow{r} &  \displaystyle\bigoplus_{e \in E} k^* \arrow{r} & \text{ ker } \left(\pic(C) \to \pic(\widehat{C}) \right) \arrow{r} & 0 .
  \end{tikzcd}
  \]
Since $k^*$ is closed under taking square roots, we have ${\ext^{1}(\autgroup,k^*) \cong 0}$ and ${\ext^{1}\left(\autgroup,\left. (\bigoplus_{i \in V} k^*) \middle/ k^* \right. \right) \cong 0}$. After applying the functor $\hom(\autgroup,-)$ to the above sequence, we get the following exact sequence:
    $$ \bigoplus_{i \in V} \autgroup \to \bigoplus_{e \in E} \autgroup \to \text{ ker } \left(\pic(C) \to \pic(\widehat{C}) \right)[2] \to 0 .$$
    We see that the first map in the above sequence is precisely the coboundary map as in Definition \ref{definition-coboundary-map}. We have an isomorphism ${\ker \left(\pic(C) \to \pic(\widehat{C}) \right)[2] \xrightarrow{f^*} H}$ and thus we get an isomorphism $H \cong \cohomology{1}{\Gamma
}{\autgroup}$ as claimed.
  \end{sloppypar}
\end{proof}

\begin{remark}\label{remark_graph_pairing}
  \begin{sloppypar}
    We have the evaluation pairing ${C^1(\Gamma,\autgroup) \times \ C_1(\Gamma,\autgroup) \xrightarrow{\graphpairing{}{-}{-}} \autgroup}$. It is easy to see that if $\gamma$ is in the image of the coboundary map $\delta$, and $\alpha$ is in the kernel of the boundary map $\partial$, then the pairing $\graphpairing{}{\gamma}{\alpha}$ is zero. Consequently, we get a pairing ${\cohomology{1}{\Gamma}{\autgroup} \times \homology{1}{\Gamma}{\autgroup} \to \autgroup}$ which is perfect (See Lemma 2.1, \cite{baker_metric_properties_of_abel_jacobi_map}).
    We denote this pairing by $\graphpairing{\Gamma}{-}{-}$.
  \end{sloppypar}
\end{remark}
} 

\section{Azumaya algebras and Weil pairing}
{
\newcommand{\brauer}{\text{Br}}
\newcommand{\azumaya}[2]{\mathcal{A}_{\{#1,#2\}}}
\newcommand{\brauermap}{\text{br}}
\newcommand{\twistedcurve}{\mathcal{X}}
\newcommand{\weilpairing}[3]{\langle #2,#3 \rangle_{#1}}
\newcommand{\curveup}{\widetilde{\mathcal{X}}}
\newcommand{\autgroup}{\integers/2\integers}
Let $X$ be a quasi-compact, separated Deligne-Mumford stack. A matrix algebra on $X$ is an $\structuresheaf_{X}\mbox{-algebra}$ which is isomorphic to the endomorphism algebra of a finite, locally free $\structuresheaf_{X}\mbox{-module}$. An Azumaya algebra on $X$ is an $\structuresheaf_X$-algebra which is \etale{} locally isomorphic to a matrix algebra (See \cite{grothendiec_brauer_I} or \cite{antieau2020brauer}). An Azumaya algebra is said to be trivial if it is isomorphic to a matrix algebra.
\begin{remark}\label{remark-end-M-iso-to-end-N}
  \begin{sloppypar}
    Suppose $M$ and $N$ are two finite, locally free $\structuresheaf_{X}\mbox{-modules}$ on $X$ such that the endomorphisms algebras $\sheafhom_{\structuresheaf_{X}}(M,M)$ and $\sheafhom_{\structuresheaf_{X}}(N,N)$ are isomorphic. Then there exists a line bundle $L$ on $X$ such that ${M \otimes_{\structuresheaf_{X}} L \cong N}$ (See \cite[\href{https://stacks.math.columbia.edu/tag/0A2K}{Tag 0A2K}]{stacks-project}).
  \end{sloppypar}
\end{remark}
\begin{sloppypar}
\noindent  The multiplication operation on Azumaya algebras is given by the tensor product over $\structuresheaf_{X}$. Two Azumaya algebras $\mathcal{A}$ and $\mathcal{B}$ are said to be Brauer equivalent, if we have an isomorphism ${\mathcal{A} \otimes_{\structuresheaf_X} T_1 \cong \mathcal{B} \otimes_{\structuresheaf_X} T_2}$ for some trivial Azumaya algebras $T_1$ and $T_2$. The Brauer group of $X$, which we denote $\brauer(X)$, is the group of equivalence classes of Azumaya algebras. Let $\mathcal{A}$ be an Azumaya algebra which is locally free of rank $n^2$. We can regard $\mathcal{A}$ as a $PGL_n$-torsor over $X$ and therefore such Azumaya algebras are classified by the group $\etalecohomology{1}{X}{PGL_n}$ (See \cite{grothendiec_brauer_I}). From the exact sequence of sheaves:
 $$ 1 \to \gmsheaf \to GL_n \to PGL_n \to 1 $$
we obtain a group homomorphism ${\brauermap_{X} : \text{Br}(X) \to \etalecohomology{2}{X}{\gmsheaf}}$ which is called the Brauer class map. The Brauer class of $\mathcal{A}$ in $\brauer(X)$ is killed by $n$ (See \cite[\href{https://stacks.math.columbia.edu/tag/0A2L}{Tag 0A2L}]{stacks-project}). The Brauer class map is injective and its image is contained in the torsion part of $\etalecohomology{2}{X}{\gmsheaf}$ (See \cite{grothendiec_brauer_I}).

  Assume that two is invertible on $X$ and let $L_1, L_2 \in \pic(X)[2]$ be two $2\mbox{-torsion}$ line bundles on $X$. Then we define the $\structuresheaf_X$-algebra $\azumaya{L_1}{L_2}$ as follows. We set
  $$ \azumaya{L_1}{L_2} = \structuresheaf_X \oplus L_1 \oplus L_2 \oplus (L_1 \otimes L_2) .$$
The multiplication on $\azumaya{L_1}{L_2}$ is defined by fixing isomorphisms $L_1^2 \cong \structuresheaf_X$, $L_2^2 \cong \structuresheaf_X$, and if $e_1, e_2$ are local sections $L_1$ and $L_2$ respectively, then we define $e_1 \times e_2 = - e_2 \times e_1$ as a section of $L_1 \otimes L_2$. It is easy to see that this defines a matrix algebra whenever $L_1$ and $L_2$ are trivial. Therefore $\azumaya{L_1}{L_2}$ is an Azumaya algebra which corresponds to a two torsion element in the Brauer group.
  \begin{remark}\label{remark-azumaya-and-pullback}
    \begin{sloppypar}
      Let $f: Y \to X$ is a morphism of quasi-compact Deligne-Mumford stacks. Let $L_1,L_2 \in \pic(X)[2]$ be $2\mbox{-torsion}$ line bundles on $X$. Then we have ${f^* \azumaya{L_1}{L_2} \cong \azumaya{f^*L_1}{f^*L_2}}$. Moreover, if ${\azumaya{L_1}{L_2} \cong \sheafhom_{\structuresheaf_{X}}(M,M)}$ for some locally free $\structuresheaf_{X}\mbox{-module}$ $M$ then we have ${\azumaya{f^*L_1}{f^*L_2} \cong \sheafhom_{\structuresheaf_{Y}}(f^*M,f^*M)}$. 
    \end{sloppypar}
  \end{remark}

Let $\twistedcurve$ be a twisted curve over $k$ on which $2$ is invertible. We use Azumaya algebras to define the Weil pairing on $\twistedcurve$. The key ingredient is the following:
\begin{lemma}\label{lemma-h2_gm_is_zero}
The cohomology group $\etalecohomology{2}{\twistedcurve}{\gmsheaf}$ is zero (See Proposition 5.2, \cite{poma2013etale}).
\end{lemma}
It follows from the above Lemma \ref{lemma-h2_gm_is_zero} that any Azumaya algebra on $\twistedcurve$ is trivial. Let $L_1,L_2 \in \pic(\twistedcurve)[2]$ be two $2\mbox{-torsion}$ line bundles on $\twistedcurve$. Then we have ${\azumaya{L_1}{L_2} \cong \sheafhom_{\structuresheaf_{\twistedcurve}}(M,M)}$ for some locally free $\structuresheaf_{\twistedcurve}\mbox{-module}$ of rank two. We define the pairing $\weilpairing{\twistedcurve}{L_1}{L_2}$ to be the degree of the line bundle $\wedge^2 M$ modulo two. We note that $M$ is unique up to a twist (See Remark \ref{remark-end-M-iso-to-end-N}) and therefore the partify of the degree of $\wedge^2 M$ is well-defined.
  \begin{definition}\label{definition-weil-pairing}
    The Weil pairing on $\twistedcurve$ is the map ${\pic(\twistedcurve)[2] \times \pic(\twistedcurve)[2] \to \autgroup}$ as defined above. We denote this pairing by $\weilpairing{\twistedcurve}{-}{-}$. 
  \end{definition}
  \begin{remark}\label{remark-smooth-curve-pairing}
    When $\twistedcurve$ is a smooth, proper curve over $k$, we know that the Weil pairing as defined above coincides with the usual Weil pairing (See Lemma 2, \cite{mumford_theta_characteristics}). Therefore we know in this case that the Weil pairing is bilinear, alternating and non-degenerate.
\end{remark}
\end{sloppypar}
} 

\section{Weil pairing on twisted curves} 
{
\newcommand{\brauer}{\text{Br}}
\newcommand{\azumaya}[2]{\mathcal{A}_{\{#1,#2\}}}
\newcommand{\brauermap}{\text{br}}
\newcommand{\twistedcurve}{\mathcal{X}}
\newcommand{\weilpairing}[3]{\langle #2,#3 \rangle_{#1}}
\newcommand{\curveup}{\widetilde{\mathcal{X}}}
\newcommand{\autgroup}{\integers/2\integers}
\newcommand{\character}[2]{\theta^{#1}_{#2}}
\newcommand{\graphpairing}[3]{\langle #2,#3 \rangle_{#1}}
\newcommand{\graphup}{\widetilde{\Gamma}}

Let $\twistedcurve$ be a twisted curve over $k$ of genus $g$ on which two in invertible. A two torsion line bundle gives an \etale{} double cover of $\twistedcurve$ which is also a twisted curve. We can use this double cover to compute the Weil pairing. For a smooth, proper curve, this idea is due to Mumford (See Lemma 2, \cite{mumford_theta_characteristics}). 
\begin{situation}\label{situation_double_cover}
Let $L_1,L_2 \in \pic(\twistedcurve)[2]$ be two $2\mbox{-torsion}$ line bundles. Let $\curveup \xrightarrow{\pi} \twistedcurve$ be the \etale{} double cover of $\twistedcurve$ given by $L_1$, and let $\tau$ denote the involution of $\curveup$ over $\twistedcurve$. Let $G$ denote the Galois group $\{id_{\curveup}, \tau \}$ of $\curveup$ over $\twistedcurve$.
\end{situation}
\begin{proposition}\label{prop-calculating-weil-pairing}
 In the above Situation \ref{situation_double_cover} we have the following:
\begin{enumerate}
\item There exists a line bundle $P$ on $\curveup$ such that 
$$ P \cong \tau^*P \otimes_{\structuresheaf_{\curveup}} \pi^*(L_2) .$$
\item For any line bundle $P$ on $\curveup$ such that the above holds, we have
$$ \weilpairing{\twistedcurve}{L_1}{L_2} \equiv \text{ degree}(P) \text{ modulo 2}.$$
\end{enumerate}
\end{proposition}
\begin{proof}
  \begin{sloppypar}
    We give a cohomological reasoning for the existence of $P$ while the rest of the proof is identical to Mumford's argument in \cite{mumford_theta_characteristics}.
We have the Hochschild-Serre spectral sequence for the Galois cover ${\curveup \xrightarrow{\pi} \twistedcurve}$ that converges to $\etalecohomology{p+q}{\twistedcurve}{\gmsheaf}$. The second page of the spectral sequence is as follows:
    \[
      \begin{tikzcd}[cramped,sep=small]
        \cohomology{2}{\curveup}{\gmsheaf[\curveup]}^{G}  & \phantom{} & \phantom{} & \phantom{} \\
        {\pic(\curveup)}^{G} \arrow[drr,close,"d^2_{0,1}"] & \cohomology{1}{G}{\pic(\curveup)} \arrow[drr,close,"d^2_{1,1}"] & \cohomology{2}{G}{\pic(\curveup)} & \phantom{} \\
        \gmsheaf(\curveup)^{G} & \cohomology{1}{G}{\gmsheaf(\curveup)} &  \cohomology{2}{G}{\gmsheaf(\curveup)} & \cohomology{3}{G}{\gmsheaf(\curveup)} \\
      \end{tikzcd}
    \]
    Since $\etalecohomology{2}{\twistedcurve}{\gmsheaf}$ is zero, we conclude that the map $d^2_{1,1}$ is injective. The assignment $\tau \to L_2$ defines a group homomorphism $G \xrightarrow{h} \pic(\twistedcurve)$. The composite group homomorphism $G \xrightarrow{\pi^* \circ h} \pic(\curveup)$ gives a $1\mbox{-cocycle}$ with values in $\pic(\curveup)$ and represents a group cohomology class in $\cohomology{1}{G}{\pic(\curveup)}$. We know that this group cohomology class lies in the kernel of the map $d^2_{1,1}$ (See Proposition 6.17, \cite{Torsors_and_gerbes}). Therefore, we conclude that the group cohomology class given by $\pi^* \circ h$ is zero. In other words, there exists a line bundle $P$ on $\curveup$ such that ${P \cong \tau^*P \otimes_{\structuresheaf_{\curveup}}\pi^*L_2}$ (See Proposition 7.1, \cite{Torsors_and_gerbes}). This proves the first part of the proposition.

    Let $P$ be any line bundle on $\curveup$ that satisfies the first part of the proposition. Let $V$ denote the $\structure_{\twistedcurve}\mbox{-module}$ $\pi_*P$. Since $\pi$ is a finite, \etale{} morphism of degree two, we conclude that $V$ is a locally $\structure_{\twistedcurve}\mbox{-module}$ of rank two. We claim that the Azumaya algebra $\azumaya{L_1}{L_2}$ is isomorphic to the matrix algebra $\sheafhom_{\structuresheaf_{\twistedcurve}}(V,V)$. Since $P$ is an $\structuresheaf_{\curveup}\mbox{-module}$, we have a map of sheaves ${\structuresheaf_{\curveup} \times P \to P}$ which gives a map         ${\pi_*(\structuresheaf_{\curveup}) \times \pi_*(P)} \to \pi_*(P)$. Since $\pi_*(\structuresheaf_{\curveup}) \cong \structuresheaf_{\twistedcurve} \oplus L_1$, we thus have an action of ${\structuresheaf_{\twistedcurve} \oplus L_1}$ on $V$. We choose an isomorphism ${\alpha: \pi^*L_2 \otimes_{\structuresheaf_{\curveup}} P \to \tau^*P}$ such that the composition:
    $$ P \cong \pi^*L_2 \otimes \pi^*L_2 \otimes P \xrightarrow{id \otimes \alpha} \pi^*L_2 \otimes \tau^*P \xrightarrow{\tau^*\alpha} \tau^* (\tau^*P) \cong P $$
    is identity. Then we have the map:
    $$ \pi^*L_2 \otimes_{\structuresheaf_{\curveup}} \left(P \oplus \tau^*P \right) \cong \left( \pi^*L_2 \otimes P \right) \oplus \left(\pi^*L_2 \otimes \tau^*P \right) \xrightarrow{\alpha \oplus \tau^* \alpha} \tau^*P \oplus P \cong P \oplus \tau^* P.$$ 
    We have ${\pi^*V \cong P \oplus \tau^*P}$ and the adjoint of the above map gives a morphism ${L_2 \otimes_{\structuresheaf_{\twistedcurve}} V \to V}$ that defines an action of $L_2$ on $V$. Altogether, this makes $V$ a module over the $\structuresheaf_{\twistedcurve}\mbox{-algebra}$ $\azumaya{L_1}{L_2}$ and gives an isomorphism ${\azumaya{L_1}{L_2} \cong \sheafhom_{\structuresheaf_{\twistedcurve}}(V,V)}$ (See Lemma 2, \cite{mumford_theta_characteristics}). Finally, we observe that 
    $$ \pi^*(\wedge^2 V) \cong \wedge^2 \pi^*(V) \cong P \otimes \tau^*P $$
and therefore we have 
    $$ 2 \text{ deg} (\wedge^2V) = \text{ deg}(P) + \text{ deg}(\tau^*P) = 2 \text{ deg}(P).$$
We conclude that the degree of $\wedge^2V$ equals the degree of $P$. This proves the second part of the proposition.
  \end{sloppypar}
\end{proof}
\begin{corollary}\label{corollary-Weil-pairing-is-bilinear-map}
The Weil pairing on $\twistedcurve$ is bilinear and alternating.
\end{corollary}
\begin{proof}
  \begin{sloppypar}
    Let $L_1,L_2,L'_2$ be $2\mbox{-torsion}$ line bundles on $\twistedcurve$. It is clear from the definition that the algebra $\azumaya{L_2}{L_1}$ is the opposite algebra of $\azumaya{L_1}{L_2}$. Therefore if ${\azumaya{L_1}{L_2} \cong \sheafhom_{\structuresheaf_{\twistedcurve}}(M,M)}$ for some locally free, finite $\structuresheaf_{\twistedcurve}\mbox{-module}$ $M$, then we have ${\azumaya{L_2}{L_1} \cong \sheafhom_{\structuresheaf_{\twistedcurve}}(M^*,M^*)}$ where $M^*$ is the dual module ${\sheafhom_{\structuresheaf_{\twistedcurve}}(M,\structuresheaf_{\twistedcurve})}$. Since ${\wedge^2 M^*}$ is isomorphic to ${\left(\wedge^2M\right)^{-1}}$, we see that ${\weilpairing{\twistedcurve}{L_1}{L_2} = - \weilpairing{\twistedcurve}{L_2}{L_1}}$. 

Let ${\curveup \xrightarrow{\pi} \twistedcurve}$ be the \etale{} double cover of $\twistedcurve$ defined by $L_1$ as in Proposition \ref{prop-calculating-weil-pairing}. Let $\tau$ denote the non-trivial involution of $\curveup$ over $\twistedcurve$. Then by Proposition \ref{prop-calculating-weil-pairing} there exist line bundles $P,P'$ on $\curveup$ such that 
    $$ \pi^*L_2 \cong \tau^*P \otimes_{\structuresheaf_{\curveup}} P^{-1} \text{ and } \pi^*L'_2 \cong \tau^*P' \otimes_{\structuresheaf_{\curveup}} P'^{-1} .$$
    Then we see that 
    $$ \tau^*(P \otimes_{\structuresheaf_{\curveup}} P') \otimes_{\structuresheaf_{\curveup}} (P \otimes_{\structuresheaf_{\curveup}} P')^{-1} \cong L_2 \otimes_{\structuresheaf_{\curveup}} L'_2.$$
    Therefore,
    \begin{align*}
      \weilpairing{\twistedcurve}{L_1}{L_2 \otimes_{\structuresheaf_{\curveup}} L'_2} = \deg(P \otimes_{\structuresheaf_{\curveup}} P') = \deg(P) + \deg(P') = \weilpairing{\twistedcurve}{L_1}{L_2} + \weilpairing{\twistedcurve}{L_1}{L'_2}
    \end{align*}
    Thus the Weil pairing is linear is the second variable. Since we have observed that ${\weilpairing{\twistedcurve}{L_1}{L_2} = - \weilpairing{\twistedcurve}{L_2}{L_1}}$, we conclude that the pairing is bilinear. Finally, the pullback $\pi^*(L_1)$ is the trivial line bundle on $\curveup$ and therefore we conclude from Proposition \ref{prop-calculating-weil-pairing} that ${\weilpairing{\twistedcurve}{L_1}{L_1} = 0}$. 
  \end{sloppypar}
\end{proof}
\begin{lemma}\label{remark_no_obstruction_for_torsors}
  In Situation \ref{situation_double_cover}, we have an exact sequence: 
$$ 0 \to \cohomology{1}{G}{\gmsheaf(\curveup)} \to \pic(\twistedcurve) \xrightarrow{\pi^*} \pic(\curveup)^G \to 0 .$$
\end{lemma}
\begin{proof}
  \begin{sloppypar}
    We claim that $\cohomology{2}{G}{\gmsheaf(\curveup)}$ is zero. If $\curveup$ is disconnected, then we have ${\gmsheaf(\curveup) \cong k^* \oplus k^*}$ on which $G$ acts by swapping the summands. Therefore $\cohomology{2}{G}{\gmsheaf(\curveup)}$ is zero. If $\curveup$ is connected, we have $\gmsheaf(\curveup) \cong k^*$ with a trivial $G$ action. Since $k^*$ is closed under taking square roots, we see that $\cohomology{2}{G}{\gmsheaf(\curveup)}$ vanishes in this case as well. We have Hochschild-Serre spectral sequence for the cover ${\curveup \xrightarrow{\pi} \twistedcurve}$ that we considered in the proof of Proposition \ref{prop-calculating-weil-pairing}. A consequence of the spectral sequence is that we have a long exact sequence in low degrees which reads as follows:
\[
  \begin{tikzcd}[column sep = small]
    0 \arrow{r} & \cohomology{1}{G}{\gmsheaf(\curveup)} \arrow{r} & {\pic}(\twistedcurve) \arrow{r}{\pi^*} & {\pic}(\curveup)^G \arrow{r} & \cohomology{2}{G}{\gmsheaf(\curveup)} \arrow{r} & \etalecohomology{2}{\twistedcurve}{\gmsheaf}.
  \end{tikzcd}
\]
Since $\cohomology{2}{G}{\gmsheaf(\curveup)}$ is zero, we have the exact sequence as stated in the lemma.
  \end{sloppypar}
\end{proof}
\begin{situation}\label{situation-notation-twistedcurve-with-normalization}
We have the following commutative diagram:
\[
\begin{tikzcd}
  \widehat{\twistedcurve} \arrow{r}{\nu} \arrow{d}{\widehat{f}} & \twistedcurve \arrow{d}{f} \\
  \widehat{C} \arrow{r}{h} & C 
\end{tikzcd}
\]
The morphism $\twistedcurve \xrightarrow{f} C$ is the morphism from the twisted curve $\twistedcurve$ to its coarse space $C$. The morphism ${ \widehat{C} \xrightarrow{h} C}$ is the normalization of $C$. The normalization $\widehat{\twistedcurve}$ of $\twistedcurve$ is the fiber product $\widehat{C} \times_{C} \twistedcurve$. We denote the dual graph of $\twistedcurve$ by $\Gamma$, its vertex set by $V$, and the edge set by $E$. For ${e \in E}$, we denote the stabilizer of the node $e$ by $G_e$. We enumerate the vertex set $V$ as $\{1,2,3,\ldots,n\}$. The conected components of $\widehat{\twistedcurve}$ are smooth orbifold curves over $k$ and we denote them by $\{\twistedcurve_i\}_{i \in V}$. The morphism ${\widehat{f}: \widehat{\twistedcurve} \to \widehat{C}}$ maps each of these to its coarse space $C_i$ which is a connected component of $\widehat{C}$.
\end{situation}
\begin{proposition}\label{prop_pairing_coming_from_coarsespace}
In Situation \ref{situation-notation-twistedcurve-with-normalization}, let  $L_1, L_2 \in \pic(C)[2]$ be two torsion line bundles on the coarse space $C$. Then we have
$$ \weilpairing{\twistedcurve}{f^*L_1}{f^*L_2} \equiv \sum_{i=1}^{n}  \weilpairing{C_i}{h^*L_1|_{C_i}}{h^*L_2|_{C_i}} \text{ modulo }2$$ 
\end{proposition}
\begin{proof}
  \begin{sloppypar}
    The Azumaya algebra $\azumaya{L_1}{L_2}$ on $C$ is trivial as $\etalecohomology{2}{C}{G_m}$ vanishes. Therefore we have ${\azumaya{L_1}{L_2} \cong \sheafhom_{\structuresheaf_{C}}(M,M)}$ for some locally free, finite $\structuresheaf_{C}$-module $M$. Then it follows (See Remark \ref{remark-azumaya-and-pullback}) that we have
    \begin{align*}
      \weilpairing{\twistedcurve}{f^*L_1}{f^*L_2} &\equiv \deg(\wedge^2 f^*M) \text{ modulo }2 \\
                                                  &= \sum_{i=1}^n \deg(\wedge^2M|_{C_i}) \\ 
                                                  &\equiv \sum_{i=1}^n \weilpairing{C_i}{h^*L_1}{h^*L_2} \text{ modulo }2
    \end{align*}
  \end{sloppypar}
\end{proof}

\begin{theorem}\label{thm_combinatorial_pairing_induced_by_Weil}
In Situation \ref{situation-notation-twistedcurve-with-normalization}, assume that the stabilizers of the nodes of $\twistedcurve$ are of even order. Let $H \subset \pic(\twistedcurve)[2]$ be the subgroup defined by ${H = \{ L \in \pic(\twistedcurve)[2] \, \vert \, \nu^*(L) \text{ is trivial.} \}}$. Then the Weil pairing on $\twistedcurve$ induces a pairing: 
$$ H \times \frac{\pic(\twistedcurve)[2]}{\pic(C)[2]} \to \autgroup. $$
Moreover, with isomorphisms ${H \cong \cohomology{1}{\Gamma}{\autgroup}}$ and ${\frac{\pic(\twistedcurve)[2]}{\pic(C)[2]} \cong \homology{1}{\Gamma}{\autgroup}}$ (See Propositions \ref{prop_two_torsion_trivial_on_normalization} and \ref{prop_exact_sequence_picard_two_torsion_of_twisted}), the above pairing can realized as the graph pairing (as in Remark \ref{remark_graph_pairing}):
$$ \cohomology{1}{\Gamma}{\autgroup} \times \homology{1}{\Gamma}{\autgroup} \xrightarrow{\graphpairing{\Gamma}{-}{-}} \autgroup .$$
\end{theorem}
\begin{proof}
  \begin{sloppypar}
    Let $L_1 \in H$ and $L_3 \in \pic(C)[2]$ be $2\mbox{-torsion}$ line bundles. Since $L_1$ is trivial on the normalization, we have ${L_1 \cong f^*L_1'}$ for some $L_1' \in \pic(C)[2]$ such that $h^*L_1'$ is trivial on each component of $\widehat{C}$ (See Proposition \ref{prop_two_torsion_trivial_on_normalization}). It follows from Proposition \ref{prop_pairing_coming_from_coarsespace} that we have
$${ \weilpairing{\twistedcurve}{L_1}{f^*L_3} = \sum_{i=1}^n \weilpairing{C_i}{h^*L_1'}{h^*L_3} = 0}.$$
Thus we see that the Weil pairing restricted to ${H \times \pic(\twistedcurve)}$ factors through a pairing 
    $$ H \times \frac{\pic(\twistedcurve)[2]}{\pic(C)[2]} \to \autgroup $$
    as claimed. 

    We use the second part of Proposition \ref{prop-calculating-weil-pairing} to prove that the above pairing is identical to the graph pairing. Let $\curveup \xrightarrow{\pi} \twistedcurve$ be the \etale{} double over given by $L_1$. Let $\tau$ denote the involution of $\curveup$ over $\twistedcurve$. We have the commutative diagram as below where ${\widehat{\curveup} = \widehat{\twistedcurve} \times_{\twistedcurve} \curveup}$ is the normalization of $\curveup$.
    \[
      \begin{tikzcd}
        \widehat{\curveup} \arrow{r}{\widetilde{\nu}} \arrow{d}{\widehat{\pi}} & \curveup \arrow{d}{\pi} \\
        \widehat{\twistedcurve}  \arrow{r}{\nu} & \twistedcurve
      \end{tikzcd}
    \]
    Since $\nu^*L_1$ is trivial, the double cover ${\widehat{\curveup} \to \widehat{\twistedcurve}}$ is trivial. Therefore, $\curveup$ has two copies of every irreducible component of $\twistedcurve$ and the components are glued together along nodes as prescribed by $L_1$. We say two irreducible components of $\curveup$ are conjugate if they are interchanged by the involution $\tau$. Let $L_2 \in \pic(\twistedcurve)[2]$ be a $2\mbox{-torsion}$ line bundle that we want to pair with $L_1$. By Proposition \ref{prop-calculating-weil-pairing}, there exists a line bundle $P$ on $\curveup$ such that 
    $$ P \otimes \tau^*(P)^{-1} \cong \pi^*(L_2) $$
and for any such line bundle $P$, the Weil pairing $\weilpairing{\twistedcurve}{L_1}{L_2}$ is given by the parity of the degree of $P$. In particular, we note that $P$ can be replaced by $P \otimes \pi^*(L)$ for any line bundle $L$ on $\twistedcurve$. In view Remark \ref{remark_edge_disjoint_cycles}, we may assume that the image of $L_2$ in $\homology{1}{\Gamma}{\autgroup}$ is a cycle (See Proposition \ref{prop_exact_sequence_picard_two_torsion_of_twisted}). That is, we assume there are distinct irreducible components $\twistedcurve_1, \twistedcurve_2, \ldots, \twistedcurve_l$ of $\twistedcurve$ and distinct nodes $e_1, e_2, \dots, e_l$  such that the following holds:
    \begin{itemize}
    \item For $1 \leq i \leq l-1$, the node $e_i$ lies on $\twistedcurve_i$ and $\twistedcurve_{i+1}$, and the node $e_l$ lies on $\twistedcurve_1$ and $\twistedcurve_l$. For brevity, we write this by saying $\{\twistedcurve_1,e_1, \twistedcurve_2, e_2, \ldots, \twistedcurve_l, e_l,\twistedcurve_1\}$ is a cycle in $\twistedcurve$. 
    \item The characters of $L_2$ are trivial at all nodes except at $\{e_1, e_2, \ldots, e_l\}$.
    \item For $1 \leq i \leq l$, the character $\character{e_i}{\twistedcurve}(L_2) \in \pic(BG_{e_i})$ is the unique non-trivial element of order two.
    \end{itemize}
Let $\alpha \in \homology{1}{\Gamma}{\autgroup}$ denote the homology class of $e_1 + e_2 + \ldots + e_l$ and let $\gamma \in \cohomology{1}{\Gamma}{\autgroup}$ denote the cohomology class corresponding to $L_1$ as given by the isomorphism in Proposition \ref{prop_two_torsion_trivial_on_normalization}. Let $\graphup$ denote the dual graph associated to the curve $\curveup$. The key observation is that  
    $$ \pi^{-1}(e_1 + e_2 + \ldots + e_l) \in \homology{1}{\graphup}{\autgroup} $$
    is one cycle of length $2l$ if $\graphpairing{\Gamma}{\gamma}{\alpha} = 1$, and is a disjoint union of two cycles of length $l$ each, if $\graphpairing{\Gamma}{\gamma}{\alpha} = 0$ (See Theorem 5.11, \cite{Jensen_Yoav_kernel_is_isotropic}).

    \textbf{Case 1: }Suppose $\graphpairing{\Gamma}{\gamma}{\alpha} = 0$. Then for $i \in \{1, 2, \ldots l\}$, there are conjugate irreducible components ${\{\twistedcurve_i^a, \twistedcurve_i^b\}}$ of $\curveup$ which are mapped isomorphically to $\twistedcurve_i$, and there are conjugate nodes $e_i^a,e_i^b$ lying above $e_i$, such that
    $$ \{\twistedcurve^a_1,e^a_1, \twistedcurve^a_2, e^a_2, \ldots, \twistedcurve^a_n, e^a_n,\twistedcurve^a_1\} \text{ and }\{\twistedcurve^b_1,e^b_1, \twistedcurve^b_2, e^b_2, \ldots, \twistedcurve^b_n, e^b_n,\twistedcurve^b_1\} $$
are cycles in $\curveup$. The line bundle $\pi^*(L_2) \cong P \otimes \tau^*(P)^{-1}$ has trivial characters at all the nodes except at $\{e_i^a,e_i^b | 1 \leq i \leq l\}$ where it has the non-trivial character of order two. Therefore, we can twist $P$ by $\pi^*(L)$ for some conveniently chosen line bundle $L$ on $\twistedcurve$ (See Lemma \ref{remark_no_obstruction_for_torsors}) to ensure the following:
    \begin{enumerate}
    \item \label{item_outside_node_trivial} If $e$ is a node of $\curveup$, and $e \notin \{e_i^a, e_i^b | 1 \leq i \leq l \}$, then the character of $P$ at $e$ is trivial. 
    \item \label{item_a_node_trivial} The characters of $P$ at the nodes $\{e_i^a | 1 \leq i \leq l \}$ are trivial. 
    \end{enumerate}
Since $\pi^*(L_2)$ is a torsion line bundle, its degree on any irreducible component of the normalization of ${\curveup}$ is zero. As we have ${P \otimes \tau(P)^{-1} \cong \pi^*(L_2)}$, we conclude that the degree of $P$ on a component is equal to its degree on the conjugate component. Thus, in order to show that the total degree of $P$ is even, it suffices to argue that the degree of $P$ on any irreducible component is integral. Since $P$ satisfies the first condition, it follows that the degree of $P$ on a component that does not contain any of the nodes ${\{e^a_i, e^b_i | 1 \leq i \leq l \}}$ is integral. Since $P$ also satisfies the second condition, the degree of $P$ on any of the components ${\{\twistedcurve_i^a | 1 \leq i \leq l \}}$ is integral. Since the components $\{\twistedcurve_i^b\}$ are conjugate to $\{\twistedcurve_i^a\}$, the degree on $P$ on any of the components ${\{\twistedcurve_i^b | 1 \leq i \leq l \}}$ is integral as well. It follows that the total degree of $P$ is even and hence we have $\weilpairing{\twistedcurve}{L_1}{L_2} = 0$. 

    \textbf{Case 2:} Suppose $\graphpairing{\Gamma}{\gamma}{\alpha} = 1$. Then for $i \in \{1, 2, \ldots l\}$, there are conjugate irreducible components ${ \{\twistedcurve_i^a, \twistedcurve_i^b\}}$ of $\curveup$ which are mapped isomorphically to $\twistedcurve_i$, and there are nodes $e_i^a,e_i^b$ lying above $e_i$ such that
    $$ \{\twistedcurve_1^a, e^a_1, \twistedcurve_2^a, e^a_2, \ldots,e^a_{l-1}, \twistedcurve_l^a, e^a_l, \twistedcurve_1^b, e^b_1, \twistedcurve_2^b, \ldots ,e_{l-1}^b, \twistedcurve_l^b, e_l^b, \twistedcurve_1^a \}$$
    is a cycle of length $2l$ in $\curveup$. As in the previous case, the line bundle $\pi^*(L_2)$ has trivial characters at all nodes except at $\{e_i^a,e_i^b | 1 \leq i \leq n\}$ where it has the non-trivial character of order two. Again by the same argument we used earlier, we can ensure that the line bundle $P$ meets the conditions \ref{item_outside_node_trivial} and \ref{item_a_node_trivial} of the previous case. Therefore the degree of $P$ on any component that does not contain the nodes ${\{e^a_i, e^b_i | 1 \leq i \leq n \}}$ is integral. From the second condition, we conclude that the degree of $P$ is integral on the components ${\{\twistedcurve_i^a | 2 \leq i \leq l\}}$. Since the degree is constant on conjugate pairs, the degree of $P$ is also integral on the components ${\{\twistedcurve_i^b | 2 \leq i \leq n\}}$. Now, the character of $P$ at $e^a_l$ is trivial but $\pi^*(L_2)$ has the non-trivial character of order two at $e^a_l$. Hence the character of $P$ at $e^b_l$ must be the non-trivial element of order two. Therefore, we conclude that the degree of $P$ on $\twistedcurve_1^a$ is a half-integer (See Remark \ref{remark_non_integral_degree}). In conclusion, the degree of $P$ on any component of ${\curveup}$ is integral and equal to its degree on the conjugate component except for the conjugate pair $\{\twistedcurve_1^a, \twistedcurve_1^b\}$. The degree of $P$ on $\twistedcurve_1^a$ equals its degree on $\twistedcurve_1^b$ and it is a half-integer. It follows that the total degree of $P$ is odd, and hence we have $\weilpairing{\twistedcurve}{L_1}{L_2} = 1$.
  \end{sloppypar}
\end{proof}
\begin{remark}\label{remark_pairing_when_not_all_are_even}
  \begin{sloppypar}
    When not all stabilizers of the nodes in $\twistedcurve$ are of even order, the induced pairing on ${H \times \frac{\pic(\twistedcurve)[2]}{\pic(C)[2]}}$ corresponds to the pairing:
    $$ \cohomology{1}{\Gamma}{\autgroup} \times \homology{1}{\Gamma'}{\autgroup} \to \autgroup $$
    where $\Gamma'$ is the modified graph obtained by deleting the edges of $\Gamma$ corresponding to the nodes with odd stabilizers (See Remark \ref{remark_edge_disjoint_cycles}). This corresponds to the graph theoretic pairing where we regard $\homology{1}{\Gamma'}{\autgroup}$ as a natural subgroup of $\homology{1}{\Gamma}{\autgroup}$ (See Remark \ref{remark_edge_disjoint_cycles}).
  \end{sloppypar}
\end{remark}
\begin{proposition}\label{prop-weil-pairing-is-non-degenerate}
The Weil pairing on $\twistedcurve$ is non-degenerate if and only if the stabilizers of the non-separating nodes are of even order or equivalently if $\lvert \pic(\twistedcurve)[2] \rvert = 2^{2g} $.
\end{proposition}
\begin{proof}
  \begin{sloppypar}
We have the exact sequence:
    $$ 0 \to \pic(C)[2] \xrightarrow{f^*}  \pic(\twistedcurve)[2] \to \homology{1}{\Gamma'}{\autgroup} \to 0 $$
where $\Gamma'$ is the modified dual graph obtained by deleting the edges corresponding to the nodes with odd stabilizers (See Remark \ref{remark_edge_disjoint_cycles}). We regard $\pic(C)[2]$ as a subgroup of $\pic(\twistedcurve)[2]$. Let $H \subseteq \pic(C)[2]$ denote the subgroup of $\pic(\twistedcurve)[2]$ consisting of those line bundles which are trivial on the normalization. Then we have $H \cong \cohomology{1}{\Gamma}{\autgroup}$ (See Proposition \ref{prop_two_torsion_trivial_on_normalization}). With the notation as in Situation \ref{situation-notation-twistedcurve-with-normalization}, we have the exact sequence:
$$ 0 \to H \to \pic(C)[2] \to \displaystyle\bigoplus_{i \in V}\pic(C_i)[2] \to 0 .$$
We know that the Weil pairing is non-degenerate on each $C_i$ which is a smooth, proper curve. Therefore, it follows from Proposition \ref{prop_pairing_coming_from_coarsespace} that the kernel of the Weil pairing restricted to ${\pic(C)[2]}$ is $H$. Since the graph pairing is perfect, we conclude that that the pairing restricted to ${H \times \pic(\twistedcurve)[2]}$ has kernel $\pic(C)[2]$ in the second factor (See Proposition \ref{thm_combinatorial_pairing_induced_by_Weil} and Remark \ref{remark_pairing_when_not_all_are_even}). It follows that the Weil pairing on $\pic(\twistedcurve)[2]$ is non-degenerate if and only if the induced pairing:
 $${H \times \frac{\pic(\twistedcurve)[2]}{\pic(C)[2]} \to \autgroup}$$
is perfect. This is true exactly when the stabilizers of the non-separating nodes are of even order (See Remark \ref{remark_pairing_when_not_all_are_even} and Remark \ref{remark_edge_disjoint_cycles}).
  \end{sloppypar}
\end{proof}
}
\section{Weil pairing in a family }
{
\newcommand{\family}{\mathscr{X}}
\newcommand{\specialfib}{\mathscr{X}_0}
\newcommand{\genfib}{\mathscr{X}_{{\eta}}}
\newcommand{\familycse}{C}
\newcommand{\specialfibcse}{C_{0}}
\newcommand{\genfibcse}{C_{\eta}}
\newcommand{\weilpairing}[3]{\langle #2,#3 \rangle_{#1}}
\newcommand{\autgroup}{\integers/2\integers}
\newcommand{\azumaya}[2]{\mathcal{A}_{\{#1,#2\}}}
\newcommand{\groupstackcompact}{\overline{M_g}^{\structuresheaf,2}}
\newcommand{\twistedcurve}{\mathcal{X}}
\newcommand{\rootscheme}[2]{\overline{#1^{1/#2}}}
\newcommand{\rootstack}[2]{#1^{1/#2}}
\newcommand{\divisors}{\text{Div}}
\newcommand{\prindivisors}{\text{Prin}}
\newcommand{\geogenfibcse}{C_{\overline{\eta}}}

Let $B$ be a strictly Heneslian local ring with the residue field $k$ such that $\spec(B)$ has a unique generic point which we denote by $\eta$. A special case of interest is when $B$ is a complete discrete valuation ring. By a family of twisted curves over a scheme we mean the notion of twisted curves due to Abramovich and Vistoli as outlined by Olsson in \cite{_Martin_olsson_log_twisted_curves} (Definition 1.2).
\begin{situation}\label{situation-family-of-curves}

  \begin{sloppypar}
    Let ${\family \xrightarrow{} \spec(B)}$ be a family of twisted curves of genus $g \geq 2$. Let ${j : \specialfib \to \family}$ denote the special fiber and ${i: \genfib \to \family}$ denote the generic fiber. Let ${\familycse \to \spec(B)}$ be the coarse space of $\family$ and we denote its special fiber and the generic fiber by $\specialfibcse$ and $\genfibcse$, respectively.
  \end{sloppypar}
\end{situation}
\begin{proposition}\label{prop_deforming_line_bundles}
  \begin{sloppypar}
    In Situation \ref{situation-family-of-curves}, let $r$ be a positive integer invertible on $\spec(B)$. Then we have the following:
    \begin{enumerate}
    \item \label{item_bijection_on_r_torsion_injection_on_generic} 
      The map ${\pic(\family)[r] \xrightarrow{j^*} \pic(\specialfib)[r]}$
      is bijective.
      \item  \label{item_dvr_injection} If $B$ is a discrete valuation ring, then the map ${\pic(\family)[r] \xrightarrow{i^*} \pic(\genfib)[r]}$ is injective. 
        \item \label{item_line_bundle_deforms} The map $\pic(\family) \xrightarrow{j^*} \pic(\specialfib)$ is surjective. 
        \end{enumerate}
  \end{sloppypar}
\end{proposition}
\begin{proof}
  \begin{sloppypar}
    These assertions essentially follow from Chiodo's constructions in \cite{chiodo_stable_twisted_and_r_spin}. Let $F$ be a line bundle on $\family$ whose relative degree is a multiple of $r$. Then there is a Deligne-Mumford stack ${\rootstack{F}{r} \to \spec(B)}$ that parameterizes $r^{th}$ roots of the line bundle $F$. An object of $\rootstack{F}{r}$ over a $B$-scheme $B'$ is a pair $(M,\lambda)$ where $M$ is a line bundle on ${\family_{B'} = \family \times_B B'}$ and ${\lambda: M^{\otimes r} \to F_{B'}}$ is an isomorphism. The stack $\rootstack{F}{r}$ is \etale{} and separated over $\spec(B)$ (Proposition 3.4, \cite{chiodo_stable_twisted_and_r_spin}). After rigidifying along the automorphisms given by the group sheaf $\mu_r$ on $B$, we have a scheme $\rootscheme{F}{r} \to \spec(B)$ whose fiber over a geometric point $q$ is ${\{L \in \pic(\family_{q}) \lvert L^{\otimes r} \cong F_{q}\}}$. We refer to \cite{chiodo_stable_twisted_and_r_spin} (Definition 3.5) and references therein. The scheme $\rootscheme{F}{r}$ is \etale{} and separated over $\spec(B)$. Moreover, if all geometric fibers $\family_q$ have $r^{2g}$ roots of $F_p$ in $\pic(\family_{q})$, then $\rootscheme{F}{r}$ is a finite, \etale{} scheme over $\spec(B)$ (Proposition 3.7, \cite{chiodo_stable_twisted_and_r_spin}). 

    We apply the above construction to the $r^{th}$ roots of the trivial line bundle $\structuresheaf_{\family}$ on $\family$. Since $B$ is strictly Henselian, and the morphism $\rootscheme{\structuresheaf_{\family}}{r} \to \spec(B)$ is \etale{} and separated, any point in the closed fiber uniquely extends to a section (See \cite{milne_etale_cohomology_book}, Theorem 4.2, (c) and (d)). This proves part \ref{item_bijection_on_r_torsion_injection_on_generic}. If $B$ is a disecrete valuation ring then we apply the valuative criterion of separatedness to the separated morphism $\rootscheme{\structuresheaf_{\family}}{r} \to \spec(B)$. This proves part \ref{item_dvr_injection}.

Let $L_0 \in \pic(\specialfib)$ be a line bundle on $\specialfib$. We say $L_0$ can be lifted to $\family$ if it lies in the image of the map $j^*$. We first argue that if $L_0$ is a pull back of a line bundle on the coarse space $\specialfibcse$ of $\specialfib$, then it can be lifted. A line bundle on the nodal curve $\specialfibcse$ is given by a Cartier divisor. Such a divisor can be deformed to a Cartier divisor on the coarse space of $\family$. We argue this is as follows. Let $p$ be a closed point in the smooth locus of $\specialfibcse$. Then there is a Zariski neighborhood $U \subset \familycse$ of $p$ such that $U \to \spec(B)$ is smooth. Since $B$ is strictly Henselian and $U$ is smooth over $B$, the section given by $p$ over the closed point of $\spec(B)$ can be extended to a section $\spec(B) \xrightarrow{s} U$. Since $U \to \spec(B)$ is smooth, the section $s$ gives a Cartier divisor in all geometric fibers of $\familycse$. Therefore it is a Cartier divisor in $\familycse$ (See \cite[\href{https://stacks.math.columbia.edu/tag/062Y}{Tag 062Y}]{stacks-project}). Thus, we see that the Cartier divisors on $\specialfibcse$ can be lifted to Cartier divisors on $\familycse$. Therefore, if $L_0 \in \pic(\specialfib)$ is a pull back of a line bundle on $\specialfibcse$, then it can be lifted. In the general case, there is a positive integer $n$ invertible on $\spec(B)$ such that $L_0^n$ is a pull back of a line bundle on $\specialfibcse$. Since $L_0^n$ can be lifted, there exists a line bundle $L$ on $\family$ whose relative degree is a multiple of $n$, and which restricts to $L_0^n$ on the special fiber. Then we have the scheme $\rootscheme{L}{n}$ which is \etale{} and separated over $\spec(B)$. Moreover, $\rootscheme{L}{n}$ admits a section over the closed point given by $L_0$. Since $B$ is Henselian, this section can be extended to a section over $\spec(B)$ (See \cite{milne_etale_cohomology_book}, Theorem 4.2). In other words, $L_0$ can be lifted to a line bundle on $\family$. 
\end{sloppypar}
\end{proof}
\begin{remark}\label{remark_r_divisible_stabs}
  \begin{sloppypar}
    Suppose in the Situation \ref{situation-family-of-curves}, the stabilizers of the nodes have orders divisible by $r$, then $\rootscheme{\structuresheaf_{\family}}{r}$ is a finite \etale{} cover of $\spec(B)$ (Proposition 3.7, \cite{chiodo_stable_twisted_and_r_spin}). Since $B$ is strictly Henselian, such a cover must be trivial (See \cite{milne_etale_cohomology_book}, Theorem 4.2). Therefore we have bijections ${\pic(\family)[r] \xrightarrow{j^*} \pic(\specialfib)[r]}$ and ${\pic(\family)[r] \xrightarrow{i^*} \pic(\genfib)[r]}$. 
  \end{sloppypar}
\end{remark}
\begin{lemma}\label{lemma_r_torsion_free_G_m_cohomology}
  \begin{sloppypar}
    In Situation \ref{situation-family-of-curves}, let $r$ be a positive integer invertible on $\spec(B)$. Then the cohomology group ${\etalecohomology{2}{\family}{\gmsheaf[\family]}}$ is free of $r\mbox{-torsion}$. 
  \end{sloppypar}
\end{lemma}
\begin{proof}
We have the Kummer exact sequence:
$$ 0 \to \mu_r \to \gmsheaf[\family] \xrightarrow{\times r} \gmsheaf[\family] \to 0 $$
of \etale{} sheaves on $\family$. A part of the long exact sequence associated to the Kummer sequence reads as follows: 
$$ \cdots \to \etalecohomology{1}{\family}{\gmsheaf[\family]} \xrightarrow{h} \etalecohomology{2}{\family}{\mu_r} \to \etalecohomology{2}{\family}{\gmsheaf[\family]} \xrightarrow{\times 2} \etalecohomology{1}{\family}{\gmsheaf[\family]} \to \cdots $$
It suffices to argue that the connecting morphism ${\etalecohomology{1}{\family}{\gmsheaf[\family]} \xrightarrow{h} \etalecohomology{2}{\family}{\mu_r}}$ is surjective. By Proper base change theorem for torsion sheaves (See \cite[\href{https://stacks.math.columbia.edu/tag/095T}{Tag 095T}]{stacks-project}), we have an isomorphism ${\etalecohomology{2}{\family}{\mu_r} \xrightarrow{j^{-1}} \etalecohomology{2}{\specialfib}{j^{-1}\mu_r} \cong \etalecohomology{2}{\specialfib}{\mu_r}}$. The connecting morphism $h$ fits in the commutative diagram below:
\[
  \begin{tikzcd}
    \etalecohomology{1}{\family}{\gmsheaf[\family]} \arrow{r}{h} \arrow{d} & \etalecohomology{2}{\family}{\mu_r} \arrow{d}{\rotatebox{90}{$\cong$}} \\
    \etalecohomology{1}{\specialfib}{\gmsheaf[\specialfib]} \arrow{r} & \etalecohomology{2}{\specialfib}{\mu_r} \arrow{r} & \etalecohomology{2}{\specialfib}{\gmsheaf[\specialfib]}
  \end{tikzcd}
\]
where the vertical map ${\etalecohomology{1}{\family}{\gmsheaf[\family]} \to \etalecohomology{1}{\specialfib}{\gmsheaf[\specialfib]}}$ is given by pullback along the morphism $j$ composed with the natural map of sheaves: ${j^{-1}\gmsheaf[\family] \to \gmsheaf[\specialfib]}$. This map corresponds to ${\pic(\family) \xrightarrow{j^*} \pic(\specialfib)}$ which is surjective by Proposition \ref{prop_deforming_line_bundles}. The bottom row is part of the long exact sequence associated to the Kummer exact sequence on $\specialfib$. Since $\etalecohomology{2}{\specialfib}{\gmsheaf[\specialfib]}$ is zero (See Lemma \ref{lemma-h2_gm_is_zero}), the map ${\etalecohomology{1}{\specialfib}{\gmsheaf[\specialfib]} \to \etalecohomology{2}{\specialfib}{\mu_r}}$ is surjective. Therefore, we conclude that the map $h$ is surjective as well. 
\end{proof}
\begin{proposition}\label{prop-pairing-good-over-family}
  \begin{sloppypar}
    In Situation \ref{situation-family-of-curves}, suppose $2$ is invertible on $\spec(B)$. Let ${L_1,L_2 \in \pic(\family)[2]}$ be two $2\mbox{-torsion}$ line bundles on $\family$. Then we have equality of Weil pairings: ${\weilpairing{\specialfib}{L_1|_{\specialfib}}{L_2|_{\specialfib}} = \weilpairing{\genfib}{L_1|_{\genfib}}{L_2|_{\genfib}}}$.
  \end{sloppypar}
\end{proposition}
\begin{proof}
The Azumaya algebra $\azumaya{L_1}{L_2}$ on $\family$ corresponds to a $2\mbox{-torsion}$ Brauer class in the cohomology group $\etalecohomology{2}{\family}{\gmsheaf}$. From the above Lemma \ref{lemma_r_torsion_free_G_m_cohomology}, we conclude that $\azumaya{L_1}{_2}$ must be a matrix algebra. In other words, we have a finite, locally free $\structuresheaf_{\family}\mbox{-module}$ $V$ on $\family$ such that $\azumaya{L_1}{L_2} \cong \sheafhom_{\structuresheaf_{\family}}(V,V)$. Since the degree the line bundle $\wedge^2V$ is constant on the fibers, we conclude that the Weil pairings are equal as claimed.
\end{proof}
\begin{corollary}\label{corollary_pairing_consistent_general_case}
  \begin{sloppypar}
    Let $Z$ be a connected Noetherian scheme on which $2$ is invertible. Let $\family \to Z$ be a family of twisted curves over $Z$. Let ${L_1, L_2 \in \pic(\family)[2]}$ be $2\mbox{-torsion}$ line bundles on $\family$. Then the Weil pairing $\weilpairing{\family_{p}}{L_1}{L_2}$ is independent of the geometric fiber $\family_{p}$ of $\family$ over $Z$. 
  \end{sloppypar}
\end{corollary}
\begin{proof}
  We may assume that $Z$ is irreducible. Let $p$ be a geometric point of $Z$ lying over $s \in Z$. We have the stalk ${\structure^{sh}_{Z,s}}$ of the \etale{} structure sheaf which is a strictly Henselian ring. Then $\family \times_{Z} \spec({\structure^{sh}_{Z,s}})$ is a family of twisted curves over $\spec({\structure^{sh}_{Z,s}})$ as in Situation \ref{situation-family-of-curves}. From the above Proposition \ref{prop-pairing-good-over-family} we conclude that the Weil pairing on $\family_{p}$ is equal to the Weil pairing on the generic fiber.
\end{proof}
Consequently, we get a well-defined Weil pairing over the moduli of twisted curves. We can use deformation methods to give an alternate argument to conclude some of the properties of the Weil pairing which we proved directly in the previous section. The crucial tool is that we can deform a twisted curve to a smooth, proper curve. 
\begin{remark}\label{remark_versal_deformations}
We have a moduli space of twisted curves of genus $g$ in which the space of smooth, proper curves forms an open, dense substack (See \cite{chiodo_stable_twisted_and_r_spin}). Let $\twistedcurve$ be a twisted curve over $k$ with nodes $\{e_i\}_{i = 1}^{m}$. Let $l_i$ be the order of the stabilizer at the node $e_i$. Then $\twistedcurve$ admits a versal deformation space given by ${R[z_1,z_2,\ldots,z_m]/(z_1^{l_1}-t_1,z_2^{l_2}-t_2,\ldots,z_m^{l_m}-t_m)}$ where $\{t_i\}$ are parameters in $R$ that define the nodes (See Theorem 2.13, \cite{chiodo_stable_twisted_and_r_spin}). A consequence of this is that there exists a family of twisted curves over a complete discrete valuation ring $B$ whose special fiber is isomorphic to $\twistedcurve$ and the generic fiber is a smooth, proper curve (See Appendix B, \cite{Baker_specialization} for the case of schematic curves). Moreover, we can ensure that \etale{} locally at a node $e$ of the special fiber, the morphism of the family to its coarse space has the form:
$$ \left[ \spec\left( B[z,w]/(zw-t) \right)  \middle/ G_{e} \right] \xrightarrow{x \to z^{l}, y \to w^{l}} \spec(B[x,y]/(xy-t^{l})) $$
where $t$ is a uniformizer of $B$, and the cyclic group $G_{e}$ acts via $z \to \zeta_l \cdot z, w \to \zeta_l^{-1} \cdot w$ for a primitive ${l}^{th}$ root $\zeta_l$ of unity.
\end{remark}
\begin{proposition}\label{prop-alternate-deformation-argument}
Let $\twistedcurve$ be a twisted cure of genus $g \geq 2$ over $k$. Then the Weil pairing on $\twistedcurve$ is bilinear and alternating. If the non-separating nodes of $\twistedcurve$ have stabilizers of even order, then it is non-degenerate.
\end{proposition}
\begin{proof}
  \begin{sloppypar}
    There exists a family of curves as in Situation \ref{situation-family-of-curves} where $B$ is a complete discrete valuation ring, the special fiber $\specialfib$ is isomorphic to $\twistedcurve$, and the generic fiber is a smooth, proper curve. We know that the Weil pairing on $\genfib$ is bilinear, alternating and non-degenerate. By Proposition \ref{prop_deforming_line_bundles} and Proposition \ref{prop-pairing-good-over-family}, we can compute the Weil pairing on $\specialfib$ by lifting the $2\mbox{-torsion}$ line bundles uniquely to $\family$ and pairing them on $\genfib$. Since ${\pic(\family)[2] \xrightarrow{i^*} \pic(\genfib)[2]}$ is injective (See Proposition \ref{prop_deforming_line_bundles}),  we conclude that the Weil pairing on $\specialfib$ is bilinear and alternating. If the non-separating nodes have stabilizers of even order, then by counting we have a bijection ${\pic(\family)[2] \to \pic(\genfib)[2]}$. Therefore we conclude that the pairing is non-degenerate on $\specialfib$ in this case.
  \end{sloppypar}
\end{proof}
\section{Connection with tropical geometry}

In this section, we look at the \emph{tropical specialization} of line bundles from curves to graphs which arises when we have an arithmetic surface. We refer to \cite{Baker_specialization} for the notions and ideas in this regard. Let $B$ be a complete discrete valuation ring with the residue field $k$ and field of fractions $K$. 
\begin{definition}
  A semi-stable arithmetic surface over $B$ is a proper, flat scheme over $\spec(B)$ such that the generic fiber is a smooth, geometrically connected curve and the special fiber is a reduced curve with at worst nodal singularities. A strongly semi-stable arithmetic surface is a semi-stable arithmetic surface such that the irreducible components of the special fiber are smooth. 
\end{definition}
By a metric graph we mean a metric space obtained by viewing the edges of a graph as line segments of prescribed lengths (See Definition 1D, \cite{Baker_specialization}). We have notions of divisors and principal divisors on a metric graph for which we refer to \cite{Baker_specialization} and the references therein. The Picard group of a metric graph is the group of divisors modulo the principal divisors. We denote the degree zero Picard group by $\pic^0$ and it is isomorphic to the real torus $\reals^{g'}/\integers^{g'}$ where $g'$ is the genus of the graph (See Theorem 2.8, \cite{baker_metric_properties_of_abel_jacobi_map}). If $\Gamma$ is a graph, then the associated metric graph is obtained by viewing each edge in $\Gamma$ as a line segment of length one. We denote it by $\Gamma$ just the same.  
\begin{situation}\label{situation-arithmetic-surface}
  \begin{sloppypar}
    Let $C$ be a regular strongly semi-stable arithmetic surface over $\spec(B)$. Let $i: \genfibcse \to \familycse$ denote the generic fiber and $j: \specialfibcse \to \familycse$ denote the special fiber. Let $\Gamma$ denote the dual graph of $\specialfibcse$. We denote the geometric generic fiber by $\geogenfibcse$. Let $r$ be a positive integer invertible on $\spec(B)$. Let $g'$ denote the genus of the graph $\Gamma$.
  \end{sloppypar}
\end{situation}
In the above Situation \ref{situation-arithmetic-surface}, we have a degree preserving homomorphism ${\pic(\familycse) \xrightarrow{\rho} \pic(\Gamma)}$. If $L \in \pic(\familycse)$, then $\rho(L)$ is a divisor supported on the vertices corresponding to the irreducible components of $\specialfibcse$ and the corresponding coefficient is the degree of $L$ restricted to the associated component. The Zariski closure of a Cartier divisor on $\genfibcse$ gives a Cartier divisor on $\familycse$. Thus we get the \emph{tropical specialization} map ${\pic(\genfibcse) \xrightarrow{trop} \pic(\Gamma)}$. The specialization map is compatible with finite field extensions of $K$ and gives a map from the geometric Picard group ${\pic(\geogenfibcse) \xrightarrow{trop} \pic(\Gamma)}$ (See Section 2C, \cite{Baker_specialization}). Moreover, the morphism $\pic^{0}(\geogenfibcse) \to \pic^{0}(\Gamma)$ is surjective onto the torsion points (or equivalently rational points) of $\pic^{0}(\Gamma)$ (See Remark A.10, \cite{Baker_specialization}).
\begin{lemma}\label{lemma_ker_of_trop_deg_zero_on_components}
  \begin{sloppypar}
    A line bundle $L_{\eta} \in \pic(\genfibcse)$ lies in the kernel of the map ${\pic(\genfibcse) \xrightarrow{trop} \pic(\Gamma)}$ if and only if there     exists a line bundle $L \in \pic(C)$ such that $i^*L \cong L_{\eta}$ and the degree of $j^*L$ is zero on every irreducible component of $\specialfibcse$.
  \end{sloppypar}
\end{lemma}
\begin{proof}
Let $D$ be the Zariski closure of a divisor on $\genfibcse$ that defines the line bundle $L_{\eta}$ on $\genfibcse$. Then $D$ is a Cartier divisor on $\familycse$. Moreover, $trop(L_{\eta}) = 0$ precisely means that there exists a vertical divisor $D'$ on $\familycse$ such that ${L = \structuresheaf_{\familycse}(D) \otimes \structuresheaf_{\familycse}(D')}$ satisfies the conditions stated in the lemma (See Section 2A, \cite{Baker_specialization} for details). 
\end{proof}
\begin{proposition}\label{prop_r_torsion_extends_to_family}
In Situation \ref{situation-arithmetic-surface}, we have the following:
  \begin{sloppypar}
    \begin{enumerate}
    \item $\lvert {\pic}(\specialfibcse)[r] \rvert = r^{2g - g'}$.
    \item For every line bundle $L_0$ in ${\pic}(\specialfibcse)[r]$, there exists a unique line bundle $L$ in ${\pic}(\familycse)[r]$ such that ${j^*L \cong L_0}$. 
    \end{enumerate}
  \end{sloppypar}
\end{proposition}
\begin{proof}
  \begin{sloppypar}
    Let $\widehat{\specialfibcse}$ denote the normalization of $\specialfibcse$. Then we have an exact sequence:
    \[
      \begin{tikzcd}
        0 \arrow{r} & \cohomology{1}{\Gamma}{\integers/ r \integers} \arrow{r} & {\pic}(\specialfibcse)[r] \arrow{r} & {\pic}(\widehat{\specialfibcse})[r] \arrow{r} & 0. 
      \end{tikzcd}
    \]
    (For $r = 2$, the see proof of Proposition \ref{prop_two_torsion_trivial_on_normalization} and the general case is similar). Thus we have ${\lvert {\pic}(\specialfibcse)[r] \rvert = 2^{2g - g'}}$. The second part is a special case of part \ref{item_bijection_on_r_torsion_injection_on_generic} of Proposition \ref{prop_deforming_line_bundles}. 
  \end{sloppypar}
\end{proof}
\begin{lemma}\label{lemma-kertrop-r-divisible}
The kernel of the morphism $\pic^0(\genfibcse) \xrightarrow{trop} \pic^0(\Gamma)$ is $r\mbox{-divisible}$. 
\end{lemma}
\begin{proof}
\textbf{Step I:} Let $\pic^0(\specialfibcse)$ denote the subgroup of $\pic(\specialfibcse)$ consisting of those line bundles whose degree is zero on each irreducible component. We first claim that $\pic^0(\specialfibcse)$ is $r\mbox{-divisible}$. If $\specialfibcse^{1}, \specialfibcse^{2}, \ldots, \specialfibcse^{n}$ are the irreducible components of the normalization $\widehat{\specialfibcse}$, then we have the exact sequence:
\[
  \begin{tikzcd}
    0 \arrow{r} &  \text{ ker } \left(\pic(\specialfibcse) \to \pic(\widehat{\specialfibcse}) \right)  \arrow{r} &  {\pic}^{0}(\specialfibcse) \arrow{r} & \displaystyle\bigoplus_{i = 1}^{n} \pic^0(\specialfibcse^{i}) \arrow{r} & 0 .
  \end{tikzcd}
  \]
The group $\pic^{0}(\specialfibcse^{i})$ is $r\mbox{-divisible}$ for each $i$. We have a description of the kernel of the above map in the proof of Proposition \ref{prop_two_torsion_trivial_on_normalization}. Since $k^*$ is closed under taking $r^{th}$ roots, it follows that the kernel above is $r\mbox{-divisible}$. Therefore we conclude that $\pic^{0}(\specialfibcse)$ is $r\mbox{-divisible}$ as well.\\
\textbf{Step II:} Let $L_{\eta} \in \pic^0(\genfibcse)$ be a line bundle on $\genfibcse$ such that $trop(L)$ is zero. By Lemma \ref{lemma_ker_of_trop_deg_zero_on_components}, there exists a line bundle $L \in \pic(\familycse)$ which restricts to $L_{\eta}$ on $\genfibcse$ and such that the degree of $j^*L$ on each irreducible component of $\specialfibcse$ is zero. We have the scheme $\rootscheme{L}{r}$ which is \etale{} and separated over $\spec(B)$ that parametrizes $r^{th}$ roots of $L$ (See the proof of Proposition \ref{prop_deforming_line_bundles}). Since $j^*L$ belongs to $\pic^{0}(\specialfibcse)$, we conclude from the first step that $\rootscheme{L}{r}$ admits a section over the closed point of $\spec(B)$. Since $\rootscheme{L}{r}$ is \etale{} over $\spec(B)$, this section can be extended to all of $\spec(B)$. In other words, $L$ admits an $r^{th}$ root in $\pic(\familycse)$ whose degree on each irreducible component of $\specialfibcse$ is zero. Consequently, $L_{\eta}$ admits an $r^{th}$ root in $\pic(\genfibcse)$ which also lies in the kernel of the $trop$ map. 
\end{proof}
\begin{corollary}\label{corollary_trop_map_surjective_on_torsion}
  \begin{sloppypar}
 The map ${\pic^{0}(\geogenfibcse)[r] \xrightarrow{trop} \pic^{0}(\Gamma)[r]}$ is surjective (See Theorem 3.1, \cite{Jensen_Yoav_kernel_is_isotropic}).
  \end{sloppypar}
\end{corollary}
\begin{proof}
  \begin{sloppypar}
    We know that ${\pic^{0}(\geogenfibcse) \xrightarrow{trop} \pic^{0}(\Gamma)}$ is surjective onto the rational points of $\pic^{0}(\Gamma)$ (See Section 2C. in \cite{Baker_specialization}). We have the exact sequence:
\[
  \begin{tikzcd}
    0 \arrow{r} & \ker (trop) \arrow{r} & {\pic}^{0}(\geogenfibcse) \arrow{r}{trop} & {\pic}^{0}(\Gamma)_{\rationals}  \arrow{r} & 0 .
  \end{tikzcd}
\]
By the above Lemma \ref{lemma-kertrop-r-divisible}, we have ${\ext^1(\integers/r\integers,\ker(trop)) = 0}$. We apply the functor $\hom(\integers/r\integers,-)$ to the above exact sequence. The conclusion follows from the resulting long exact sequence in cohomology.
  \end{sloppypar}
\end{proof}
\begin{proposition}\label{prop_ker_of_trop_r_iff_specializes}
  \begin{sloppypar}
    A line bundle $L_{\eta}$ in $\pic(\genfibcse)[r]$ lies in the kernel of the tropical specialization ${\pic(\genfibcse) \xrightarrow{trop} \pic(\Gamma)}$ if and only if there exists a line bundle $L$ in $\pic(\familycse)[r]$ such that $j^*(L) \cong L_{\eta}$.
  \end{sloppypar}
\end{proposition}
\begin{proof}
  \begin{sloppypar}
    By counting, we see that the kernel of ${\pic(\genfibcse)[r] \xrightarrow{trop} \pic(\Gamma)[r]}$ has cardinality at most $r^{2g-g'}$. By Proposition \ref{prop_r_torsion_extends_to_family}, we conclude that we have ${\lvert  \pic(\familycse)[r] \rvert = r^{2g-g'} }$. By part \ref{item_dvr_injection} of Proposition \ref{prop_deforming_line_bundles}, we know that the map ${\pic(\familycse)[r] \xrightarrow{j^*} \pic(\genfibcse)[r]}$ is injective and by Lemma \ref{lemma_ker_of_trop_deg_zero_on_components}, its image lies in the kernel of the $trop$ map. This completes the proof. 
  \end{sloppypar}
\end{proof}
\begin{remark}\label{remark_trop_specialization_vs_algebraic_specialization}
  We can interpret the above Proposition \ref{prop_ker_of_trop_r_iff_specializes} by saying that an $r\mbox{-torsion}$ line bundle in $\pic(\genfibcse)$ tropically specializes to zero if and only if it algebraically specializes (in the Zariski topology of the scheme $\rootscheme{\structuresheaf_{\familycse}}{r}$) to some $r\mbox{-torsion}$ line bundle on the special fiber. 
\end{remark}
\begin{proposition}\label{prop_ker_trop_isotropic}
  \begin{sloppypar}
    Suppose in Situation \ref{situation-arithmetic-surface}, we have $r = 2$ and that $\specialfibcse$ is totally degenerate. That is, the irreducible components of $\specialfibcse$ are all genus zero curves. Then the kernel of the map ${\pic^{0}(\geogenfibcse)[2] \xrightarrow{trop} \pic^{0}(\Gamma)[2]}$ is isotropic for the Weil pairing on $\geogenfibcse$. (See Proposition 4.3, \cite{Jensen_Yoav_kernel_is_isotropic}).
  \end{sloppypar}
\end{proposition}
\begin{proof}
  \begin{sloppypar}
  Since $\specialfibcse$ is totally degenerate, the genus of the graph $\Gamma$ is $g$ and the Weil pairing on $\pic(\specialfibcse)[2]$ is identically zero (See Proposition \ref{prop_pairing_coming_from_coarsespace}). By Proposition \ref{prop_r_torsion_extends_to_family}, we have an isomorphism ${\pic(\familycse)[2] \xrightarrow{j^*} \pic(\specialfibcse)[2]}$. By Proposition \ref{prop_ker_of_trop_r_iff_specializes}, we know that the kernel of the tropical specialization is precisely the restriction of $\pic(\familycse)[2]$ to $\genfibcse$. Since the Weil pairing is consistent in a family (See Proposition \ref{prop-pairing-good-over-family}), the conclusion follows. 
  \end{sloppypar}
\end{proof}
\begin{sloppypar}
Now we look at the situation when we have an arithmetic surface where the special fiber is a twisted curve. Suppose have a family of twisted curves as in Situation \ref{situation-family-of-curves}, where $B$ is a complete discrete valuation ring and the generic fiber $\genfib$ is a smooth curve. The coarse space $C \to \spec(B)$ is a semi-stable arithmetic surface. However, it is important to note that $C$ is \emph{not} regular at the nodes which are stacky. If $e$ is a node of $C$ whose stabilizer has order $r$, then an \etale{} local neighborhood of $C$ at $e$ is isomorphic to $\spec(B[x,y]/(xy - t^r))$ for some $t$ in the maximal ideal of $B$ (See part (ii), Proposition 2.2, \cite{_Martin_olsson_log_twisted_curves}). Since $t^r$ is not a uniformizer of $B$, we see that the $C$ fails to be regular at $e$. Let $\pi$ denote a uniformizer of $B$, and let $n$ be the positive integer such that the ideals $(t)$ and $(\pi^n)$ in $B$ are equal. Then a scheme regular at $e$, is obtained by a sequence of blow ups that modify the special fiber by introducing $nr-1$ number of rational bridges at $e$. By performing a similar operation of blow ups at all the nodes, we obtain a regular scheme $C^{reg}$.
\end{sloppypar}
\begin{proposition}\label{prop_divide_edges_into_r_parts}
  \begin{sloppypar}
 Let $\Gamma$ be a graph of genus $g$ with the vertex set $V$. We regard $\Gamma$ as a metric graph where each edge is assigned length one. For each edge, we introduce $r-1$ vertices placed equidistantly on it so that it is divided into $r$ equal parts. Let $V^{reg}$ be the union of $V$ together with these vertices. Then there are $r^{2g}$ linearly in-equivalent divisors supported on $V^{reg}$ that form the subgroup ${\pic}^{0}(\Gamma)[r]$. 
  \end{sloppypar}
\end{proposition}
\begin{proof}
  \begin{sloppypar}
    We can conceive of a twisted curve $\specialfib$ over $k$ of genus $g$ such that the following holds:
    \begin{itemize}
    \item The stabilizer of every node is a cyclic group of order $r$.
    \item The dual graph of $\specialfib$ is isomorphic to $\Gamma$.
    \end{itemize}
There exists a family of curves $\family$ over a complete discrete valuation ring $B$ as in Remark \ref{remark_versal_deformations} where the special fiber is isomorphic to $\specialfib$ and the generic fiber is smooth. Let ${C \to \spec(B)}$ denote the coarse space of $\family$. Since the stabilizers of the nodes are cyclic groups of order $r$, we have $\lvert \pic(\specialfib)[r] \rvert = r^{2g}$ (See Proposition \ref{prop_exact_sequence_picard_two_torsion_of_twisted} for $r=2$, and the general case is similar. See Corollary 3.1, \cite{chiodo_stable_twisted_and_r_spin}). We have ${\pic(\family)[r] \cong \pic(\specialfib)[r]}$ and the pull back map ${\pic(\family)[r] \to \pic(\genfib)}$ is injective (See Proposition \ref{prop_deforming_line_bundles}). Thus we have ${\lvert \pic(\genfib)[r] \rvert = r^{2g}}$. The \etale{} local structure of the coarse space $C$ at the nodes is as outlined in Remark \ref{remark_versal_deformations}. A consequence of this is that $C^{reg}$, which is obtained by a sequence of blowups as discussed above, has $r-1$ rational bridges for each node of $C$. Therefore the dual graph of $C^{reg}$ is isomorphic to $\Gamma$ (up to scaling) but with a larger vertex set $V^{reg}$. The generic fiber of $C^{reg}$ is isomorphic to $\genfib$ and therefore we have ${\pic(\genfibcse)[r] = r^{2g}}$. Therefore, we deduce that the tropical specialization ${\pic(\genfibcse)[r] \xrightarrow{trop} \pic(\Gamma)[r]}$ is surjective (See Corollary \ref{corollary_trop_map_surjective_on_torsion}). Since the image of this morphism is supported on $V^{reg}$, we conclude the result of the proposition.
  \end{sloppypar}
\end{proof}
\begin{remark}\label{remark_combinatorial_sanity_check}
  By having stabilizers only for the non-separating nodes, we see that the above Proposition \ref{prop_divide_edges_into_r_parts} also holds when we only divide the non-separating edges into $r$ equal parts (See Remark \ref{remark_edge_disjoint_cycles}). Also, we note that this is straightforward to prove combinatorially so that Proposition \ref{prop_divide_edges_into_r_parts} serves more as a consistency check and another illustration of the interplay between combinatorial tropical geometry and algebraic geometry of curves. 
\end{remark}
} 

\bibliography{tropical_bib} 
\bibliographystyle{ieeetr}
\end{document}